\documentclass{amsart}
\usepackage{amssymb,latexsym}
\usepackage{amsmath,amscd}
\usepackage{graphicx}
\usepackage{fullpage}
\usepackage{url}
\DeclareGraphicsExtensions{.pdf,.png,.jpg,.eps}
\usepackage{epstopdf}
\theoremstyle{plain}
\usepackage{fullpage}
\usepackage{multirow}
\newtheorem{thm}{Theorem}[section]
\newtheorem{prop}[thm]{Proposition}

\newtheorem{cor}[thm]{Corollary}

\theoremstyle{definition}

\theoremstyle{remark}

\usepackage{hyperref}
\usepackage{cleveref}
\usepackage{amsthm}
\usepackage{amsfonts}
\usepackage{multirow}
\usepackage{url}
\usepackage{array}

\newcommand\N{\mathbb{N}}
\newcommand\Z{\mathbb{Z}}

\numberwithin{equation}{section}

\begin{document}
\title{Taking the risk out of \emph{RISK}: Conquer Odds in the Board Game \emph{RISK}.}
\author{Sam Hendel, Charles Hoffman, Corey Manack, and Amy Wagaman}
\begin{abstract} 
Dice odds in the board game \emph{RISK} were first investigated by Tan, fixed by Osbourne, and extended by Blatt. We generalized dice odds further, varying the number of sides and the number of dice used in a single battle. We show that the attacker needs two more than $86\%$ of the defending armies to have an over $50\%$ chance to conquer an enemy territory. By normal approximation, we show that the conquer odds transition rapidly from low chance to high chance of conquering around the $86\%+2$ threshold. 
\end{abstract}
\maketitle
\section{Introduction}
\emph{RISK} is a board game introduced by Parker Brothers in $1957$, where players compete with the objective of conquering the world. 
The game dynamics offer a rich yet tractable venue for complex interactions to unfold. We view \emph{RISK} as a toy model for agents competing for resources; indeed, the Air Force recently used \emph{RISK} to model a prototypical wargame \cite{Lee12}. The majority of the game consists of battles for individual territories on the game board. Because battles are fought by rolling dice, wise decision-making relies on understanding the chances of conquering an enemy territory. We refer to such probabilities as ``conquer odds''. 
\noindent Conquer odds in \emph{RISK} were first investigated by Tan \cite{Ta97}. However, that derivation contained a discrepancy that was corrected by Osborne \cite{Os03} and extended by Blatt \cite{Bl02} (see \cite{PW14} for a recent overview). In each of \cite {Ta97},\cite{Os03},\cite{Bl02}, conquer odds were calculated using Markov chains; our methods are combinatorial. Many \emph{RISK}-like websites exist online and some of these sites allow players to modify, e.g., the number of territories, connections between those territories, and sides of dice rolled during battle. Inspired by those sites, we extend single attack odds to an arbitrary number of sides for attack and defense (cf. \ref{sec:eo}). In section \ref{sec:co} we derive a practical formula for conquer odds (cf. \ref{sec:co}) and make several observations from this formula.     

\section{Background, Statement of Results}
\label{sec:intro}
Those unfamiliar with the rules of \emph{RISK} are directed to Tan's article \cite{Ta97} for a clear introduction. A typical game board has $42$ territories occupied by armies. In a single round of \emph{RISK}, players take turns in the attacker role, placing armies on controlled territories according to how much of the board is controlled by the attacker at the beginning of attacker's the turn. Then the attacker may elect to attack from an occupied territory to any adjacent enemy territory. In a single attack, the attacker decides whether to roll $1$, $2$, or $3$ $6$-sided dice and the defender can roll $1$ or $2$ dice, contingent on the number of armies contained within their respective territories. The top die on either side are compared, and if both sides rolled at least two dice, the next highest die are compared. For each comparison, if the attacker die is larger then the defender loses an army, otherwise the attacker loses an army. This concludes a single attack. The attacker may repeat an attack from that territory until one unit remains, or may declare an attack from another controlled territory to an adjacent enemy territory. The attacker \emph{conquers} an enemy territory once all defending armies are eliminated, thereby taking control of that territory. The game is over when one player controls all $42$ territories on the board after eliminating all opponents; this usually takes several rounds. \newline

For the rest of the paper, an \emph{engagement} is a single attack and a {\emph{battle}} between a single pair of adjacent territories is a sequence of consecutive engagements that ends when one side is defeated. The \emph{Conquer odds} of a battle are the chances that the attacker conquers the enemy territory. Calculating conquer odds rest on determining probabilities of unit losses in an engagement. In \cite{Ta97}, \cite{Os03}, engagement odds for $6$-sided dice were considered as well as conquer odds.  In \cite{Bl02}, engagement odds were extended to $s$-sided dice, $s \in \N$.  We generalize engagement odds by allowing $m$ $a$-sided dice for the attacker, $n$ $d$-sided dice for the defender $(m,n,a,d\in\N )$. For any $k\in\N$ we choose, with $0<k\leq\min\{a,d\}$, an engagement is resolved by ranking the top $k$ dice of either side and comparing pairwise: maximum to maximum, $2$nd largest to $2$nd largest, and so on. For each comparison, we remove one defending army if the defender's die is less than the attacker die, else we remove one attacking army. So ties are a win for the defender, and $k$ total armies are removed after an engagement. Table \ref{table:eo} in section \ref{sec:eo} specializes the engagement odds to each possible engagement in {\emph{RISK}} $(k=1,2)$. 

\section{Conquer Odds}
\label{sec:co}
With engagement odds in hand \ref{sec:eo}, we ask: ''What is the probability that $\hat{A}$ attacking armies defeats $\hat{D}$ defending armies?'' We first mention a few conventions to ease calculations. As before, an engagement consists of a single attack, roll of the dice, resolution, and loss of units. In \emph{RISK}, the attacker has the option to stop attacking after each engagement; we will only consider battles. Another rule of \emph{RISK} is the decreasing number of dice a player can roll in subsequent low-unit engagements $\hat{A}=3,2,1$ or $\hat{D}=1$ (c.f. \cite{Ta97}). From Table \ref{table:eo}, one could quickly verify that it is always advantageous to roll with the maximum number of dice permitted. So we assume $\hat{A}> 4,\hat{D}>2$, where sides should roll with the maximum number of dice permitted. Thus engagement odds remain the same in a battle until $\hat{A}\leq 3$ or $\hat{D}= 1$. Instead of dealing with $\hat{A},\hat{D}$ directly, let $A$ be the number of units the attacker can lose until forced to roll less dice on the subsequent engagement, $D$ be the number of units the defender can lose until forced to roll with less dice on the subsequent engagement. If we assume $k$ armies $(k=1,2)$ are lost per battle and the attacker cannot attack with the last unit then $\hat{A} = A+(k+1)$, $\hat{D} = D+k$. Thus, a {\emph{virtual battle}} of $A$ attackers against $D$ defenders is a sequence of engagements that terminates once $A\leq 0$, or $D\leq 0$. We declare that $A$ attackers \emph{virtually conquers} (VC) $D$ defenders once $D\leq 0,A>0$, and a declare a \emph{virtual loss} if $A\leq 0,D>0$. Indeed, VC odds calculate the probability that the defender loses the opportunity to roll with the highest number of dice permitted before the attacker does. We shall see that a battle that terminates when $A=0,D=0$ can be included or ignored to obtain an estimate of VC odds by sums of random variables. One could easily adapt the virtual formulas to obtain exact formulas and table \ref{table:VCAC} compares VC odds to exact conquer odds when $(k=2)$. One benefit of VC odds will be the derivation of the $86\%+2$ rule which is easy to state and matches closely to the $50\%$ cutoff for actual conquer odds.  
\subsection{One army at stake per engagement.}
\label{sec:onearmy}
To illuminate the derivation, consider the simpler case that only one army can be lost per engagement $(k=1),$ then $A=\hat A -2,\hat D -1 = D$.  To simplify notation, let $p$ be the probability that the defender loses one army in an engagement, $q=1-p$. Then $A$ attackers VC $D$ defenders if the attacker wins $D$ engagements before losing $A$ engagements. This is nothing more than the cumulative distribution of a negative binomial distribution:
\begin{equation}
\label{eq:VC1}
\Pr(\text{$A$ attackers VC $D$ defenders}) = p^{D}\sum_{k=0}^{A-1}{{D+k-1}\choose {D-1}}q^k.
\end{equation}
The probability mass function 
\[p(q,D,A)={{D+A-1}\choose {D-1}}p^Dq^{A}\]
has mean $\mu_A = qD/p$, and standard deviation $\sigma_A = \sqrt{qD}/p$.
Further, the discrete (left) differential
$$p(q,D,A)-p(q,D,A-1) = p^Dq^{A-1}\frac{(D+A-2)!}{(D-1)!(A-1)!}\left(\frac{D+A-1}{A}q-1\right),$$
changes sign from positive to negative across the threshhold $A^* = (D-1)q/p<\mu_A$. This means that near $A^*$, VC odds increase the fastest; adding additional units near $A^*$ yields the most ``bang for the buck''. 
We explain a visual method for calculating VC that will be useful when $k=2$ armies are lost per engagement in the next section. If we denote by $L$ the length of the battle, then it is clear $\min\{A,D\}\leq L < A+D$. Instead of terminating the battle once all armies on one side are lost, engage $A+D$ times, allowing for negative units. These extended battles are enumerated by lattice paths that start at $(0,0)$, terminate on the line $y=(a+d)-x$. These paths take a unit step north when the attacker loses one army and a unit step east if the attacker kills one army. For $0\leq k\leq A+D,$ the probability that an extended battle terminates at $(k,(A+D)-k)$ is
$$b(A+D,p,k) = {{A+D}\choose k}p^kq^{(A+D)-k}.$$
Extended battles of $A+D$ engagements are binomially distributed along lattice points in the first quadrant that lie on the diagonal $y=A+D-x$. In an ordinary battle, the attacker wins when a lattice paths terminates on the conquer line $x=D$ prior to crossing the defeat line $y=A$. In an extended battle, the attacker wins when a lattice path crosses the conquer line $x=d$ prior to crossing the defeat line $y=A$. Each lattice path that terminates on $y=(A+D)-x$ crosses exactly one of the lines $x=D$, $y=A$ with one exception. Lattice paths terminating at $(A,D)$ represent extended battles of mutual destruction (MD). By Pascal's formula:
\begin{align}
\Pr(\text{MD in $(A+D)$ turns, $A$ att, $D$ def}) &= {{A+D}\choose D}p^Dq^A\notag \\
&= {{A+D-1}\choose {A-1}}p^Dq^A + {{A+D-1}\choose {D-1}}p^Dq^A \label{MD2D}
\end{align}
The left term of \eqref{MD2D} consists of MD that take a final step north; these are the MD's where the defender lost all its units first and thus contributes to VC odds. Combining the two contributions two VC odds yields
\begin{equation}
\label{eq:VC2}
\Pr(\text{$A$ attackers VC $D$ defenders}) = {{A+D-1}\choose {A-1}}p^Dq^{A}+\sum_{k=0}^{A-1}{{A+D}\choose k}p^{(A+D)-k}q^{k}
\end{equation}
We now show that the two formulas for VC odds \eqref{eq:VC1}, \eqref{eq:VC2}, are equal.
\begin{prop}
\label{latticepatheq}
For all $A>0,D>0$, the probability that $A$ attackers VC $D$ defenders is 
\begin{equation}
\label{almostbin}
p^{D}\sum_{k=0}^{A-1}{{D+k-1}\choose {D-1}}q^k= {{A+D-1}\choose {A-1}}p^Dq^{A}+\sum_{k=0}^{A-1}{{A+D}\choose k}p^{(A+D)-k}q^{k}
\end{equation}
\end{prop}
\begin{proof}
Our proof is bijective. Observe,
\begin{align*}
p^{D}\sum_{k=0}^{A-1}{{D+k-1}\choose {D-1}}q^k &= p^{D}\sum_{k=0}^{A-1}{{D+k-1}\choose {D-1}}q^k1^{A-k}\\
								&= p^{D}\sum_{k=0}^{A-1}{{D+k-1}\choose {D-1}}q^k(p+q)^{A-k}\\
&= p^{D}\sum_{k=0}^{A-1}{{D+k-1}\choose {D-1}}q^k\sum_{l=0}^{A-k}{{A-k}\choose l}p^{A-k-l}q^{l}\\
&= \sum_{k=0}^{A-1}\sum_{l=0}^{A-k}{{D+k-1}\choose {D-1}}{{A-k}\choose l}q^{k+l}p^{D+A-(k+l)}
\end{align*}
Consider one term in this double summation 
\[{{D+k-1}\choose {D-1}}{{A-k}\choose l}q^{k+l}p^{D+A-(k+l)}.\] Notice, If we let $h=k+l$ be the height of the terminal vertex of an extended battle, then this term represents the probability that an extended battle terminates at $((A+D)-h,h)$, taking $k$ north steps prior to first passage through $x=D$, then $l$ north steps after. When $h<A$, the probability that an extended battle terminates at $((a+d)-h,h)$ is partitioned over all (k,l) pairs satisfying $h=k+l$, and when $h=A$ only the extended battles reaching $x=D$ before reaching $y=A$ contribute to VC. These terms are all accounted for in the double summation. Reindexing,  
\begin{align*}
\sum_{k=0}^{A-1}\sum_{l=0}^{A-k}{{D+k-1}\choose {D-1}}{{A-k}\choose l}q^{k+l}p^{D+A-(k+l)} &= \sum_{h=0}^{A-1}\sum_{l+k=h}{{D+k-1}\choose {D-1}}{{A-k}\choose l}p^{A-k-l}q^{k+l}\\
&+p^D\sum_{k=0}^{A-1}{{D+k-1}\choose {D-1}}q^kq^{A-k}\\
&= {{D+A-1}\choose {A-1}}p^Dq^A+\sum_{h=0}^{A-1}{{A+D}\choose h}q^hp^{A+D-h}\qedhere \\
\end{align*}
\end{proof}
\noindent Prop \ref{latticepatheq} is interesting in its own right, expressing the negative binomial distribution as a binomial distibution plus the additional term ${{D+A-1}\choose {A-1}}p^Dq^A$, and 
\[0<{{D+A-1}\choose {A-1}}p^Dq^A<{{D+A}\choose {A}}p^Dq^A.\]
yields the fundamental estimate
\begin{equation}
\label{almostbin2} 
\sum_{h=0}^{A-1}{{A+D}\choose h}q^hp^{A+D-h}< \Pr(\text{$A$ attackers VC $D$ defenders}) < \sum_{h=0}^{A}{{A+D}\choose h}q^hp^{A+D-h}
\end{equation}

\subsection{Two armies at stake per engagement}
\label{sec:twoarmies}
Those familiar with \emph{RISK} know that the most common engagement has two armies are at stake $(k=2)$. This happens when the attacker rolls $m= 3$ dice and the defender rolls $n= 2$ dice so we specialize to this case. Let $\hat A - 3= A, \hat D -2= D$. Following \cref{sec:onearmy}, we outline an enumeration strategy. Picture the part of $\Z^3$ that lies in the nonnegative first octant. Call a {\emph{lattice path}} in $(\Z^+)^3$ any path from $(0,0,0)$ to $(x,y,z)\in(\Z^+)^3$ that takes unit steps in the $e_1,e_2,$ or $e_3$ directions. After an engagement, take a unit step in the $e_1$ direction if the defender lost $2$ armies, a unit step in the $e_2$ direction if each lost $1$ army and a unit step in the $e_3$ direction if the attacker lost $2$ armies. Under this identification, a lattice path records the history of a particular battle. Tracing any lattice path terminating at $(x,y,z)$ shows, of the $2x+2y+2z$ armies lost,  $2x+y$ were defenders, and $y+2z$ were attackers.  So calculating $\Pr(\text{$A$ VC $D$})$
amounts to finding the probability that a lattice path of length at most $(A+D)/2$ crosses the ``conquer plane'' $2x+y=D$ before reaching the ``defeat plane'' $y+2z=A$. If we let $X$ be the random variable for the number of defenders lost in an engagement, given by
\begin{center}
\begin{tabular}{c|c}
X& probability \\ \hline
2 & $p = \Pr(\text{Defender loses $2$})$\\
1 & $q = \Pr(\text{Defender loses $1$})$\\
0 & $r = \Pr(\text{Defender loses $0$})$\\
\end{tabular}
\end{center}
then the probability that a lattice path terminates at $(x,y,z)$ is
\begin{equation}
\label{terminal}
p(x,y,z) = {{x+y+z}\choose {x+y}}{{x+y}\choose y}p^xq^yr^z=\frac{(x+y+z)!}{x!y!z!}p^xq^yr^z.
\end{equation}
Thus
$$\Pr(\text{$A$ VC $D$}) =\sum_{\substack{
            2x+y \geq D \\
            y+2z < A}}p(x,y,z)$$
which suffices to determine VC odds with a simple program. We are interested in a simple rule that is easily to remember. To this end, notice that
$X$ has mean $\mu = E[X]=2p+q$, and variance 
\begin{equation}
\label{variance}
\sigma^2 = p(2-(2p+q))^2+q(1-(2p+q))^2+r(0-(2p+q))^2=pq+4pr+qr.
\end{equation} 
We assume $A+D$ is even, leaving $A+D$ odd for the reader. Similar to the $k=1$ case in \cref{sec:onearmy}, we consider extended (virtual) battles of $(A+D)/2$ total engagements and loan units to a defeated player if necessary. Thus one side may have negative units after an extended battle is over.  Each lattice path of length $(A+D)/2$ terminates at a point on the plane $2x+2y+2z = A+D$ and crosses exactly one of the planes $2x+y = D$, $y+2z=A$, with the exception of lattice paths that terminate on the line of mutual destruction (MD) given by \[\{(x,y,z)\in (\Z^+)^3\mid 2x+y = D \text{ and }y+2z=A\}.\] If we let $X_i$ be random variable of defenders lost at engagement $i$, then the number of defenders lost in an extended battle of $(A+D)/2$ engagements is precisely 
$$Z_{(A+D)/2}=X_1+X_2+\ldots+X_{(A+D)/2}.$$
Whether we include the MD odds $\Pr(Z_{(A+D)/2}=D)$ provides the fundamental estimate of VC odds:
\begin{equation}
\label{almostnormal}
\Pr(Z_{(A+D)/2}>D)<\Pr(\text{$A$ VC $D$})<\Pr(Z_{(A+D)/2}\geq D)
\end{equation}
The benefits of \ref{almostnormal} is that classical results about sums of independent and identically distributed random variables can be brought to bear on VC odds. To begin, fix $D$ and allow $A$ to vary. The mean of $Z_{(A+D)/2}$ is $D$ exactly when $A^*$ is the solution of $$D = (A^*+D)\mu /2=(2p+q)(A^*+D)/2,$$ or 
\begin{equation}
\label{mean}
A^*=\frac{q+2r}{2p+q}D=\frac{2-\mu}{\mu}D.
\end{equation}
Rewriting \eqref{mean} as
\[ \frac{A^*}{D}=\frac{(2-\mu)(A^*+D)/2}{\mu(A^*+D)/2}.\]
allows us to interpret this result in \emph{RISK}: $A^*$ attackers are expected to VC $D$ defenders at the threshold of mutual destruction, when the expected proportion of attackers lost to defenders lost after $(A^*+D)/2$ engagements is equal to the proportion of attackers to defenders at the beginning of the battle.\\
Now we find an interval of $(A_1,A_2)$ around $A^*$ that quantifies the change in VC odds. For any $s>0$, solving
$$D=(A+D)\mu/2\pm s\sigma\sqrt{(A+D)/2}$$
for $A$ yields two solutions  
\begin{equation}
\label{eq1}
A_1 = A^*+\frac{s\sigma}{\mu^2}( s\sigma-\sqrt{4D\mu+(s\sigma)^2})
\end{equation}
\begin{equation}
\label{eq2}
A_2 = A^*+\frac{s\sigma}{\mu^2}( s\sigma+\sqrt{4D\mu+(s\sigma)^2}) 
\end{equation}
$A_1$ represents the number of attackers required so that $D$ is $s$ standard deviations below $E[Z_{(A_1+D)/2}]$, and $A_2$ the number of attackers for $D$ to be $s$ standard deviations above $E[Z_{(A_2+D)/2}]$. Notice $A_1<A^*<A_2$ and $(A_1,A_2)$ is symmetric about $A^*+(s\sigma/\mu)^2$.\\
Thus, \eqref{almostnormal},\eqref{eq1},\eqref{eq2} allow us to approximate the change in VC odds near $A^*$ by direct calculation for low-unit battles and normal approximation for big battles. Indeed, the Central Limit Theorem says $Z_{(A+D)/2}$ is well-approximated (in distribution) by the normal random variable
\[ Z=N\left(\frac{(A+D)}{2}\mu,\sqrt{\frac{(A+D)}{2}}\sigma\right),\]
when $A+D$ is sufficiently large. A reasonable criterion for normal approximation is to choose $A+D$ large enough so that $0$ is less than $3$ standard deviations below $E[Z]=\frac{(A+D)\mu}{2}$ and $2(A+D)$ is more than $3$ standard deviations above $E[Z]$. This yields the inequality
\begin{equation}
\label{eq:bigenough}
A+D>\max\left\{\frac{18\sigma^2}{\mu^2},\frac{9\sigma^2}{2(1-\mu/2)^2}\right\}
\end{equation}
whence we make the approximation
\begin{equation}
\label{normalapprox}
\Pr(\text{$A$ attackers VC $D$ defenders})\approx \Pr(Z_{(A+D)/2}>D)
\end{equation}
Normal approximation allows for decisive answers to two major questions related to VC odds.
\begin{enumerate} 
\item By \eqref{mean} $A^*$ is the (approximate) $50\%$ cutoff for VC odds when $A^*+D$ is sufficiently large \eqref{eq:bigenough}. 
\item When $A_1+D$ is large enough \eqref{eq:bigenough}, so is $A_2+D$, and by the empirical rule, 
\begin{align*}
& \Pr(A_2\text{ VC }D) - \Pr(A_1\text{ VC }D) \\ 
\approx& \Pr(Z_{A_2+D}> D)-\Pr(Z_{A_1+D}> D)\\ 
\approx & \Pr(Z> -s)-\Pr(Z> s) = \Pr(|Z|< s)
\end{align*}
Observe, 
\begin{equation}
\label{eq:changeco}
A_2-A_1 = \frac{2s\sigma}{\mu^2}\sqrt{4D\mu+(s\sigma)^2}.
\end{equation}
 If the attacker has $A_1$ armies, then a percentage increase of
\begin{equation}
\label{eq:length}
\frac{A_2}{A_1}-1= \frac{2}{\frac{2-\mu}{2s\sigma}\sqrt{D\mu}-1}\approx\frac{4s\sigma}{(2-\mu)\sqrt{D\mu}}
\end{equation}
in attacking armies yields an increase in VC odds approaching $\Pr(|Z|< s)$. Notice that this percentage tends to $0$ as $D\to\infty$
\end{enumerate}
\section{Back to \emph{RISK} once more} 
We specialize the results of \cref{sec:twoarmies} to the rules of ordinary \emph{RISK}, where the attacker rolls three six-sided dice $(a=6,m=3)$ and the defender rolls two six-sided dice $(d=6,n=2)$. Engagement odds were correctly stated in \cite{Os03} as
\begin{center}
\begin{tabular}{c|c}
$x$ & $\Pr(X=x)$ \\ \hline
2 & $p = 2890/7776\approx .3717$\\
1 & $q = 2611/7776\approx.3357$\\
0 & $r = 2275/7776\approx.2926$\\
\end{tabular}
\end{center}
Alternatively, one could specialize \eqref{twoatt} and \eqref{twodef1} to get the table above. By \eqref{variance},
\[\mu = 2p+q = \frac{8391}{7776}\approx 1.08, \text{ and }\]
\[\sigma^2 =pq+4pr+qr=\frac{4420535}{6718464}\approx.657968\]
In view of \eqref{almostnormal}, \eqref{normalapprox}, we approximate $Z_{(A+D)/2}$  by the normal distribution when
\[A+D>\max\{10.17,13.97\}. \]
Whence \eqref{mean} yields
\[A^* = \frac{2-\mu}{\mu}D = \frac{7161}{8391}\approx .853D\]
Thus $A^* = \frac{7161}{8391}D$ are expected to VC $D$ defenders. For large battles $(A+D> 14)$ normal approximation says that $A^*$ armies have an over $50\%$ chance to VC $D$ defenders. Recalling $A^*=\hat{A^*}-3$,$D=\hat{D}-1$ covert attackable units to actual units at the outset of a battle,  we have the rule
\begin{center}
An attacker with at least $\hat{A^*} = \frac{7161}{8391}(\hat{D}-1) +3$ armies to has an over $50\%$ chance of VC $\hat{D}$ defenders.
\end{center}
A conservative rule would then be for the attacker to have $86\%$ of the defending armies plus two additional units.  
For low-unit battles, table \ref{table:VCAC} compares VC odds and actual conquer odds (AC). We see that VC odds are a close match to AC odds, but conservative. One can check that the $86\% +2$ rule is a serviceable rule for small unit battles as well.   
\begin{table}
\begin{tabular}{|c|c|c|c|c|c|c|c|c|}
\label{table:VCAC}
 Defenders  $(\hat{D})$   & CO & $20\%$ & $30\%$ & $40\%$ & $50\%$ & $60\%$ & $70\%$ & $80\%$ \\  \hline
\multirow{2}{*}{2}  &VC&{4}&{4}&{5}&{5}&{5}&{5}&{6}\\
\ \ \ \! &AC &3&3&4&4&4&5&6 \\ \hline
\multirow{2}{*}{3} &VC&{4}&{4}&{5}&{6}&{6}&{6}&{8}\\
\ \ \ \! &AC &3&3&4&5&5&6&7 \\ \hline
\multirow{2}{*}{4} &VC&{5}&{5}&{6}&{6}&{7}&{8}&{9}\\
\ \ \ \! &AC &4&4&5&6&6&7&8 \\ \hline
\multirow{2}{*}{5} &VC&{5}&{6}&{6}&{7}&{8}&{9}&{10}\\
\ \ \ \! &AC &4&5&5&6&7&8&9 \\ \hline
\multirow{2}{*}{6} &VC&{6}&{7}&{7}&{8}&{9}&{10}&{11}\\
\ \ \ \! &AC &5&6&6&7&8&9&10 \\ \hline
\multirow{2}{*}{7} &VC&{6}&{7}&{8}&{9}&{10}&{11}&{12}\\
\ \ \ \! &AC &5&6&7&8&9&10&11 \\ \hline
\multirow{2}{*}{8} &VC&{7}&{8}&{9}&{10}&{11}&{12}&{13}\\
\ \ \ \! &AC &6&7&8&9&10&11&12 \\ \hline
\multirow{2}{*}{9} &VC&{8}&{8}&{10}&{10}&{12}&{13}&{14}\\
\ \ \ \! &AC &7&8&9&10&11&12&13\\ \hline
\multirow{2}{*}{10} &VC&{8}&{9}&{10}&{11}&{13}&{14}&{15}\\
\ \ \ \! &AC &7&8&10&10&12&13&15 \\ \hline
\multirow{2}{*}{11} &VC&{9}&{10}&{11}&{12}&{13}&{15}&{16}\\
\ \ \ \! &AC&{8}&{9}&{10}&{11}&{13}&{14}&{16}\\ \hline
\multirow{2}{*}{12} &VC&{9}&{11}&{12}&{13}&{14}&{16}&{17}\\
\ \ \ \! &AC&{9}&{10}&{11}&{12}&{14}&{15}&{17}\\ \hline
\multirow{2}{*}{13} &VC&{10}&{12}&{13}&{14}&{15}&{17}&{18}\\
\ \ \ \! &AC&{9}&{11}&{12}&{13}&{15}&{16}&{18}\\ \hline
\multirow{2}{*}{14} &VC&{11}&{12}&{13}&{15}&{16}&{18}&{19}\\
\ \ \ \! &AC&{10}&{11}&{13}&{14}&{15}&{17}&{19}\\ \hline
\multirow{2}{*}{15} &VC&{12}&{13}&{14}&{16}&{17}&{18}&{20}\\
\ \ \ \! &AC&{11}&{12}&{13}&{15}&{16}&{18}&{20}\\ \hline
\end{tabular}
\caption{VC odds versus AC odds. Each cell lists the minimum number of attacking armies $(\hat{A})$ required to have an $n\%$ chance of conquering $\hat{D}$ defenders.}
\end{table}
\begin{center}
\begin{figure}
\scalebox{1.25}{\includegraphics{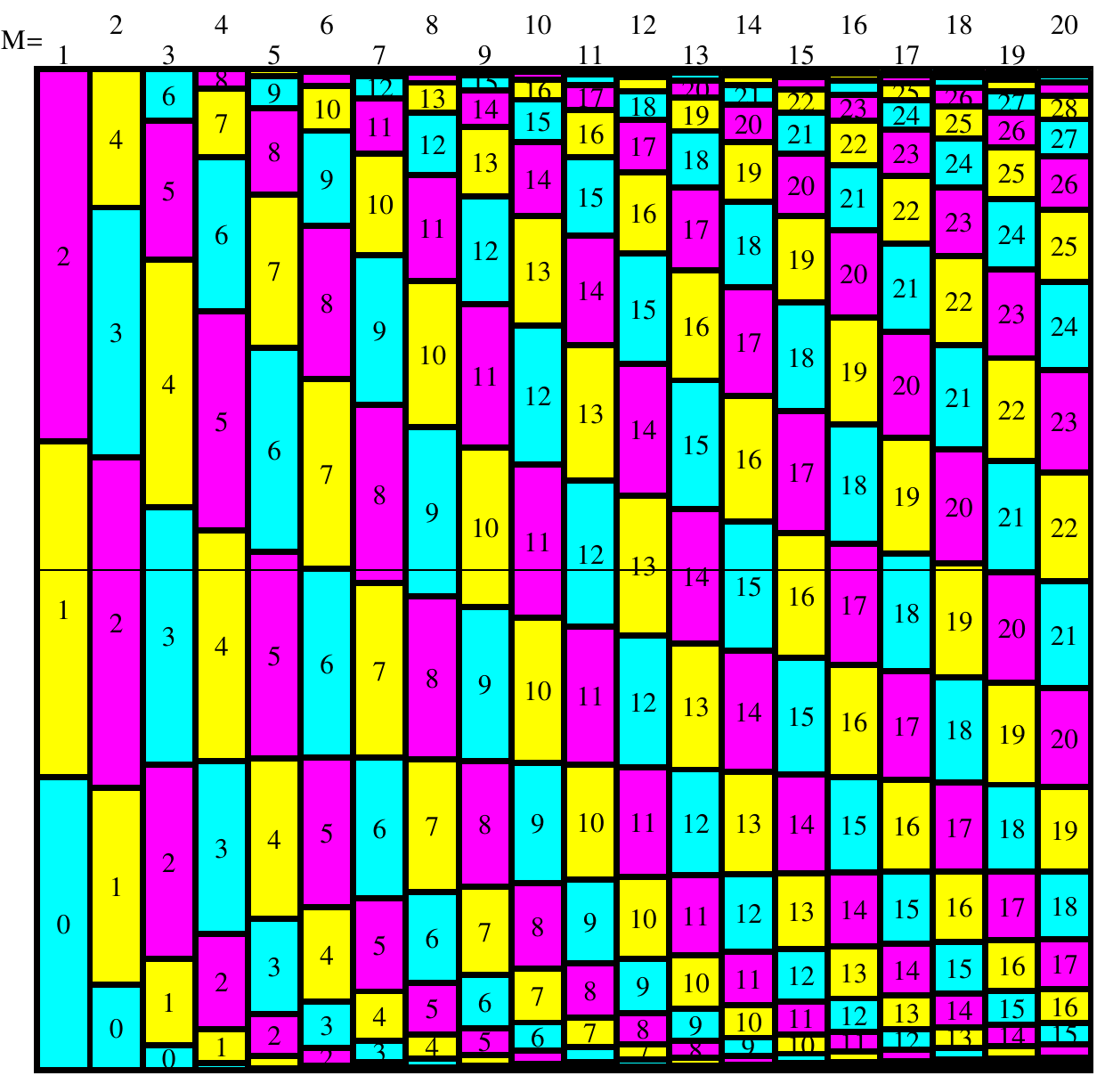}}
\caption{Column $M$ is the distribution of $Z_M=X_1+\cdots+X_M$, where $M=(A+D)/2$. The height of box $x$ in column $M$ represents the probability that an attacker kills $x$ armies in $M$ attacks. The height of the stack of boxes above box $x$ is the probability that an attacker kills more than $x$ armies in $M$ attacks. In each column, we see that box D crosses the $50\%$ line near $2M-D = .86D\ (=A^*)$, yeildiing the $86\%+2$ rule.}
\end{figure}
\end{center}
Conquer odds exhibit a phase transition around the $86\%+2$ threshold.  Fix $s>0$; normal approximation \eqref{almostnormal} and equation \label{eq:changeco} show that the interval capturing a $\Pr(|Z|<s)$ change in VC odds has length $O(\sqrt{D})$. Further, for low-unit battles, one can inspect table \ref{table:VCAC} directly to see that a transition from low VC odds to high VC odds also occurs over a small interval around $A^*$ relative to $D$.
We conclude with several implications of the $86\%+2$ rule:
\begin{itemize}
\item The $50\%$ cutoff $A^*$ is an unstable tipping point for conquer odds. Any percent deviation from $A^*$ results in a rapid departure from $50\%$ as the units are scaled up by a positive factor. For example, $10$ attackers have a $51.3\%$ chance to (virtually) conquer $9$ armies, but $100$ attacking armies have a $92.3\%$ chance of (virtually) conquering $90$ defenders. This makes the $86\% +2$ rule essential when deciding to attack. Moreover, the precise rule $\hat{A^*}=(7161/8391)(\hat{D}-1)+3$ becomes less pedantic as the battalions grow in size. 
\item In a cold war scenario, armies grow proportionally (based on bonus) between opponents with no attacks. The player with a bonus of at least $86\%$ of an opponent gains an overall attacking advantage on that opponent that increases each round.
\item When attacking a chain of $T$ connected territories, the attacker has to leave at least $1$ unit behind when transferring units into the conquered territory. If we assume full aggression, then the attacker needs approximately
\[\hat{A^*} = 0.86(\hat{D} - T-1)+3+T\] 
units to virtually conquer $\hat{D}$ defenders spread out over $T$ connected territories.
\item In a two player game, we obtain the counterintuitive result that the player with less units could be the likely victor. In particular, if the defender occupies $T$ connected territories and the attacker has one large battalion of at least $0.86(\hat{D} - T-1)+3+T$ units, the attacker has an over $50\%$ chance of victory. 
\item The proportion $7161/8391 = (2-\mu)/\mu$ is the ratio of expected attacking armies lost to expected defenders lost in a single engagement. This proportion says that each unit is $8391/7161-1 \approx 17\%$ more lethal in the attacker role than on defense.  
\end{itemize} 
\section{Conclusion}
We applied two approximations: VC odds as an approximation to conquer odds, and normal approximation to VC odds, yielding the $86\% +2$ rule. The method of normal approximation could be used to yield asymptotic results for generalized engagements, e.g., the attacker rolls four dice and the defender rolls three dice per engagement. The calculation of VC odds in the two-armies-at stake case \cref{sec:co} is a generalization of the so called ``problem of points'' \cite{Go14} whose solution set probability theory in motion. The method of normal approximation could be used to obtain an asymptotic solution for this generalized problem.  

Now that we understand the true nature of conquer odds, we can delve more deeply into the strategy behind \emph{RISK}. Questions range from mathematical to psychologocial, including:
\begin{itemize}
\item How much does initial placement determine a winner?
\item Does luck play a factor?
\item Can we quantify Aggressive/Vengeful/Passive play?
\item What percent of an enemy's stake in the game must you destroy to guarantee retaliation?
\item Can we determine good play? Do good players adhere to the $86\% +2$ rule?
\item Alliances? Cards?
\end{itemize}
We hope that this paper illuminated the $86\% +2$ rule, as well as its importance in \emph{RISK} strategy.  When coupled with the points above, one has in \emph{RISK} a rich gameplay that many have enjoyed for many years.    
\section{Appendix: Engagement Odds}
\label{sec:eo}
\noindent Let $X_i$ $(0\leq i\leq m)$ be i.i.d. random variables taking values in $\{1, 2, ... , a\}$, $Y_j$ $(0\leq j\leq n)$ be i.i.d. random variables taking values $\{1, 2, ... , d\}$. We think of $X_i$ as the face showing up on the attacker's $i$-th die, $Y_j$ as the face showing up on the defender's $j$-th die.  Let $\{X_i\}_m=\{X_1, X_2, .... , X_m\}$ be the collection of faces showing up on $m$ attack dice, $X^{(i)} = \max_i \{X_1, X_2, ...., X_m\}$ be the $i$th largest die from $\{X_i\}_m$. Define $\{Y_j\}_n, Y^{(j)}$ similarly. To compare from a collection of random variables similar to \emph{\emph{RISK}}, we introduce the comparison operator $>_{k,l}$, defined as follows: $\{X_i\}_m >_{k,l} \{Y_j\}_n$ if exactly $l$ of the $k$ comparisons $$X^{(1)} > Y^{(1)}, X^{(2)} > Y^{(2)}, \ldots , X^{(k)} > Y^{(k)}$$ are true. We derive the generalized engagement odds promised in \ref{sec:intro}, that is, 
$$\Pr\left(\{X_i\}_m >_{k,l} \{Y_j\}_n\right),$$ for various $k,l,m,n$.  We will only consider the cases $k=1,2$. The technicality in deriving these formulas stems from reconciling ties; In \emph{RISK}, ties go to the defender, and $>_{k,l}$ adopts this convention. Reconciling ties is compounded as the number of comparisons increases and we leave $k\geq 3$ cases for future research. When $X_i,Y_j$ are {\emph{continuous}} distributions, $\Pr(\text{at least one tie})=0$  and one can find a closed formula for all admissible $m,n,k,l$ (\cite{HM15}). The upcoming proofs of Propositions \ref{OneAttacker}, \ref{TwoAttacker}, \ref{TwoDefender} are standard, conditioning on $Y^{(1)},\ldots,Y^{(k)}$ then summing over the possible values $Y^{(1)},\ldots,Y^{(k)}$ can take. This enumeration scheme separates into two cases, $a>d$ and $a\leq d$. We are only dealing with $1$ or $2$ comparisons, so let $y_1$, $y_2$ be the respective values of $Y^{(1)}$,  $Y^{(2)}$. 
\begin{prop}
\label{OneAttacker}
For $m,n \in \N$, 
$$\Pr\left(\{X_i\}_m>_{1,1} \{Y_j\}_n\right) =\sum_{y=1}^{\min\{a,d\}}\frac{(a^m-y^m)(y^n-(y-1)^n)}{a^{m}d^{n}}$$
\end{prop}

\begin{proof}
Evaluating 
\begin{equation}
\label{eq:cond1} 
\Pr\left(\{X_i\}_m>_{1,1} \{Y_j\}_n\right) = \sum_y \Pr\left(X^{(1)} > y\right) \Pr\left(Y^{(1)}=y\right)
\end{equation}  amounts to calculating a typical summand, then summing over suitable $y$. Consider first, the case $a > d$. Clearly, $\Pr\left( Y^{(1)} = y\right) = 0$ unless $1\leq y \leq d$, where we have
\begin{equation}
\label{cond1}
\Pr\big(Y^{(1)}=y\big) = \Pr\left(Y^{(1)} \leq y\right) - \Pr\left( Y^{(1)} \leq y-1 \right) = \frac{y^n-(y-1)^n}{d^n}
\end{equation}
Furthermore, observe that
\begin{equation}
\label{cond2}
\Pr\big(X^{(1)} > Y^{(1)} | Y^{(1)} = y\big) = 1- \Pr\left( X^{(1)} \leq y \right) =  \frac{a^m-y^m}{a^m}
\end{equation}
is also nonzero for all $1\leq y \leq d = \min\{a,d\}$. Thus, the proposition holds in this case. \\
Now, assume $a \leq d$. Notice that $$\Pr\big(X^{(1)} > Y^{(1)} | Y^{(1)} = y\big) $$  is nonzero only when $1 \leq y \leq a -1$. By virtue of formula \eqref{cond2}, we may sum equation \eqref{eq:cond1} from $y=1$ to $a=\min\{a,d\}$ and substitute into the formulas (\ref{cond1}), (\ref{cond2}) as stated.   
\end{proof} 
\begin{cor}
\label{OneDefender}
For $m,n \in \N$,
$$\Pr\big(\{X_i\}_m>_{1,0} \{Y_j\}_n\big) =1-\sum_{y=1}^{\min\{a,d\}}\frac{(a^m-y^m)(y^n-(y-1)^n)}{a^{m}d^{n}}$$
\end{cor}
\begin{proof}
This result follows from Prop. \ref{OneAttacker}, and the observation that $$\Pr\big(\{X_i\}_m>_{1,0} \{Y_j\}_n\big)+\Pr\big(\{X_i\}_m>_{1,1} \{Y_j\}_n\big)=1.$$
\end{proof}

\begin{prop}
\label{TwoAttacker}
For $m,n \in \N$ and $m,n \geq 2$, 
\begin{align}
\label{twoatt}
\Pr\big(\{X_i\}_m>_{2,2} \{Y_j\}_n\big) &= \sum_{y_1=2}^{\min \{a,d\} } \sum_{y_2=1}^{y_{1}-1} \frac{n\big(y_2^{n-1}-(y_2-1)^{n-1}\big)\big(a^m-y_1^m-m(a-y_1)(y_2)^{m-1}\big)}{a^m d^n} \nonumber\\
& \quad + \sum_{y_{1}=1}^{\min\{a,d\}}\frac{\big( y_1^n-(y_1-1)^{n}-n(y_1-1)^{n-1} \big) \big(a^m+(m-1)y_1^m-(am)(y_1)^{m-1} \big)}{a^m d^n}
\end{align}
\end{prop}
\begin{proof}
Adhering to our counting strategy, we sum the conditional probabilites over all values the pair $Y^{(1)}$, $Y^{(2)}$ can assume, to obtain
\begin{align*}
\Pr\big(\{X_i\}_m>_{2,2} \{Y_j\}_n\big) = \sum_{y_1=2}^{\min \{a,d\} }& \sum_{y_2=1}^{y_{1}-1}\Pr\left(X^{(1)}>y_1,X^{(2)}>y_2 \right)\Pr\left(Y^{(1)}=y_1,Y^{(2)}=y_2\right)\\
+ \sum_{y_1=1}^{d}& \Pr\left(X^{(1)}>y_1 ,X^{(2)}>y_1 \right)\Pr\left(Y^{(1)}=y_1,Y^{(2)}=y_1\right)\\
\end{align*}
We begin with the case $a \geq d$, let $1 \leq y_2 \leq y_1 \leq d$. Recall that one should treat $y_i$ as a value of $Y^{(i)}$. Notice, when $y_2 < y_1$,
\begin{align*} 
\Pr\big(Y^{(1)} = y_1 \text{ and } Y^{(2)}=y_2\big) & = \Pr\big(Y_j=y_1 \text{ for some } j \text{ and } \max_{i \ne j}\{Y_i\}=y_2 \big)\\
& = \frac{n\big(y_2^{n-1}-(y_2-1)^{n-1}\big)}{d^n}.
\end{align*}
 When $y_1=y_2$,
\begin{align*}
\Pr\big( & Y^{(1)} = y_1 \text{ and } Y^{(2)}=y_1\big)\\[1 mm]
& = \Pr\big(Y^{(1)} \leq y_1 \big)-\Pr\big(Y^{(1)} \leq (y_1-1) \big)-\Pr\big(Y^{(1)}=y_1 \text{ and } Y^{(2)} \leq (y_1-1) \big)\\[1 mm]
& = \frac{y_1^n-(y_1-1)^{n}-n(y_1-1)^{n-1}}{d^n}.
\end{align*}
For the pair $X^{(1)}$ and $X^{(2)}$, we compute that for $1 \leq y_2 \leq y_1 \leq d$
\begin{align*}
\Pr\big(& X^{(1)} > y_1 \text{ and } X^{(2)} > y_2 \big)\\
&= 1 - \Pr \big( X^{(1)} \leq y_1\big)-\Pr\big( X^{(1)} > y_1 \text{ and } X^{(2)} \leq y_2 \big)\\
&= \frac{a^m-y_1^m-m(a-y_1)(y_2)^{m-1}}{a^m}.
\end{align*}
Therefore, we have computed that
\begin{align*}
\Pr\big(& \{X_i\}_m>_{2,2} \{Y_j\}_n\big) \\
&= \sum_{y_1=2}^{d} \sum_{y_2=1}^{y_{1}-1} \frac{n\big(y_2^{n-1}-(y_2-1)^{n-1}\big)\big(a^m-y_1^m-m(a-y_1)(y_2)^{m-1}\big)}{a^m d^n}\\
& \quad + \sum_{y_{1}=1}^{d}\frac{\big( y_1^n-(y_1-1)^{n}-n(y_1-1)^{n-1} \big) \big(a^m+(m-1)y_1^m-(am)(y_1)^{m-1} \big)}{a^m d^n}
\end{align*}
Substituting $d=\min\{ a,d\}$ affirms formula (\ref{twoatt}) in this case.\\
Finally, we note that when $a < d$, $\min\{ a,d\} = a$, and $\Pr\big( X^{(1)} > y_1 \text{ and } X^{(2)} > y_2 \big)=0$ whenever $y_1>a$. Thus, nonzero contributions to the sum can occur only when $1\leq y_1\leq a$. The derivation of summands and the range of $y_2$ are identical to the $d\leq a$ case, thus formula (\ref{twoatt}) holds for this case as well.
\end{proof}
\begin{prop}
\label{TwoDefender}
For $m,n \in \N$ and $m,n \geq 2$, 
\begin{align}
\label{twodef1}
\Pr\big(&\{X_i\}_m>_{2,0} \{Y_j\}_n\big) = \nonumber \\
&= \sum_{y_1=2}^{d} \sum_{y_2=1}^{y_{1}-1} \frac{n\big(y_2^{n-1}-(y_2-1)^{n-1}\big)\big(m(y_1-y_2)y_{2}^{m-1}+y_{2}^{m} \big)}{a^m d^n}\\
& \quad + \sum_{y_{1}=1}^{d}\frac{\big( y_1^n-(y_1-1)^{n}-n(y_1-1)^{n-1} \big) \big(y_{1}^{m} \big)}{a^m d^n} \nonumber
\end{align}
when $a \geq d$ and
\begin{align}
\label{twodef2}
\Pr\big(&\{X_i\}_m>_{2,0} \{Y_j\}_n\big) = \nonumber \\
&= \sum_{y_1=2}^{a} \sum_{y_2=1}^{y_{1}-1} \frac{n\big(y_2^{n-1}-(y_2-1)^{n-1}\big)\big(m(y_1-y_2)y_{2}^{m-1}+y_{2}^{m} \big)}{a^m d^n} \nonumber \\
& \quad + \sum_{y_{1}=1}^{a}\frac{\big( y_1^n-(y_1-1)^{n}-n(y_1-1)^{n-1} \big) \big(y_{1}^{m} \big)}{a^m d^n} \\
& \quad + \sum_{y_{1}=a+1}^d \sum_{y_2=1}^a \frac{n\big(y_2^{n-1}-(y_2-1)^{n-1}\big)\big(m(a-y_2)y_{2}^{m-1}+y_{2}^{m} \big)}{a^m d^n}\nonumber \\
& \quad +  \frac{a^m d^n-n(d-a)a^{n+m-1}-a^{m+n}}{a^m d^n}\nonumber
\end{align}
when $a \leq d$.
\end{prop}

\begin{proof}
As always, we first restrict to the case $a \geq d$ and consider which values $Y^{(1)}$, $Y^{(2)}$ can assume.  For each pair of values $1 \leq y_2 \leq y_1 \leq d$,  
\begin{align*}
\Pr\big(\{X_i\}_m>_{2,0} \{Y_j\}_n\big) = \sum_{y_1=2}^d & \sum_{y_2=1}^{y_{1}-1}\Pr\left(X^{(1)}\leq y_1,X^{(2)}\leq y_2 \right)\Pr\left(Y^{(1)}=y_1,Y^{(2)}=y_2\right)\\
+ \sum_{y_1=1}^{d}& \Pr\left(X^{(1)}\leq y_1 ,X^{(2)}\leq y_1 \right)\Pr\left(Y^{(1)}=y_1,Y^{(2)}=y_1\right).\\
\end{align*}
As the calculations of $$\Pr\left(Y^{(1)}=y_1,Y^{(2)}=y_2\right),$$ $$\Pr\left(Y^{(1)}=y_1,Y^{(2)}=y_1\right)$$ are identical to those contained in the proof of proposition \ref{TwoAttacker}, it only remains to calculate
\begin{align*}
& \Pr\big(X^{(1)} \leq y_1 \text{ and } X^{(2)} \leq y_2 \big) \\  
= &\Pr\left( y_2 < X^{(1)} \leq y_1 \text{ and } X^{(2)} \leq y_2\right) + \Pr\left( X^{(1)} \leq y_2 \right) \\
= & \frac{m(y_1-y_2)y_{2}^{m-1}+y_{2}^{m}}{a^m}.
\end{align*}
Thus, by summing over all permissible pairs $y_1$ and $y_2$, we have that
\begin{align*}
\Pr\big(& \{X_i\}_m>_{2,0} \{Y_j\}_n\big) \\
&= \sum_{y_1=2}^{d} \sum_{y_2=1}^{y_{1}-1} \frac{n\big(y_2^{n-1}-(y_2-1)^{n-1}\big)\big(m(y_1-y_2)y_{2}^{m-1}+y_{2}^{m} \big)}{a^m d^n}\\
& \quad + \sum_{y_{1}=1}^{d}\frac{\big( y_1^n-(y_1-1)^{n}-n(y_1-1)^{n-1} \big) \big(y_{1}^{m} \big)}{a^m d^n}
\end{align*}
which is exactly formula (\ref{twodef1}).\\
Next, we consider the case where $a \leq d$.  To proceed, we must consider three cases for the values $Y^{(1)},Y^{(2)}$ can take.  The first is  $1 \leq y_2 \leq y_1 \leq a$, which is exactly formula (\ref{twodef1}) except for the change in bound $d \to a$.  The second case is  $1 \leq y_2 \leq a < y_1 \leq d$.  For this case, we note that
\begin{align*}
\Pr\big(X^{(1)} \leq y_1 \text{ and } X^{(2)} \leq y_2\big) &= \Pr\big( X^{(1)} \leq y_2 \big) + \Pr\big(y_2 < X^{(1)} \leq a \text{ and } X^{(2)} \leq y_2 \big) \\
&= \frac{y_2^m+m(a-y_2)y_2^{m-1}}{a^m}
\end{align*}
In the third case $ a< y_2 \leq y_1 \leq d$, notice  
\begin{align*}
\Pr\big(X^{(1)} \leq a \text{ and } X^{(2)} \leq a\big) = 1,
\end{align*}
with
\begin{align*}
\Pr\big(a < Y^{(1)} \leq d \text{ and } a < Y^{(2)} \leq d \big)& = 1-\Pr\big(a < Y^{(1)} \leq d \text{ and } Y^{(2)} \leq a\big) - \Pr\big(Y^{(1)} \leq a\big)\\
& = 1-\frac{n(d-a)a^{n-1}}{d^n}-\frac{a^{n}}{d^n}\\
& = \frac{a^m d^n-n(d-a)a^{n+m-1}-a^{m+n}}{a^m d^n}
\end{align*}  Collecting the terms contributed from each case, we conclude 
\begin{align*}
\Pr\big(&\{X_i\}_m>_{2,0} \{Y_j\}_n\big) = \\
&= \sum_{y_1=2}^{a} \sum_{y_2=1}^{y_{1}-1} \frac{n\big(y_2^{n-1}-(y_2-1)^{n-1}\big)\big(m(y_1-y_2)y_{2}^{m-1}+y_{2}^{m} \big)}{a^m d^n}\\
& \quad + \sum_{y_{1}=1}^{a}\frac{\big( y_1^n-(y_1-1)^{n}-n(y_1-1)^{n-1} \big) \big(y_{1}^{m} \big)}{a^m d^n}\\
& \quad + \sum_{y_{1}=a+1}^d \sum_{y_2=1}^a \frac{n\big(y_2^{n-1}-(y_2-1)^{n-1}\big)\big(m(a-y_2)y_{2}^{m-1}+y_{2}^{m} \big)}{a^m d^n}\\
& \quad +  \frac{a^m d^n-n(d-a)a^{n+m-1}-a^{m+n}}{a^m d^n}\qedhere
\end{align*}
\end{proof}
With equations \ref{twodef1}, \ref{twoatt}, in hand, the $k=2$, $l=1$ probability can be calculated handily, using \begin{equation} \label{oneandone}
\Pr\big( \{X_i\}_m>_{2,1} \{Y_j\}_n\big) = 1 - \Pr\big( \{X_i\}_m>_{2,2} \{Y_j\}_n\big) - \Pr\big( \{X_i\}_m>_{2,0} \{Y_j\}_n\big) \end{equation}

\noindent With the aid of classical summation formulas, Propositions \ref{OneAttacker}, \ref{TwoAttacker}, \ref{TwoDefender}, Corollary \ref{OneDefender}, and \eqref{oneandone} can be specialized to the six probability distributions arising in the game \emph{RISK}, when the attacker rolls $a$-sided dice and the defender rolls $d$-sided dice. The results are summarized in table \ref{table:eo}.

\begin{table}[t]
\label{table:eo}
\centering
\renewcommand{\arraystretch}{1.19}
\begin{tabular}{|c|c|c|c|c|c|c|}
\hline
\multicolumn{7}{|c|}{$m= \#$ att dice, $n= \#$ def dice, $a =$ sides per att die, $d =$ sides per def die}\\ \hline
$m$   & $n$ & $k$ & $l$ & Prob & case & Specialization of: \\ \hline
\multirow{2}{*}{$1$} & \multirow{2}{*}{$1$} & \multirow{2}{*}{$1$} & \multirow{2}{*}{$1$} & $\frac{2a-d-1}{2a}$ & $a \geq d$ & \multirow{2}{*}{Prop. \ref{OneAttacker}}   \\  \cline{5-6}
 &  &  &  & $\frac{a-1}{2d}$ & $a \leq d$ & \\ \hline
\multirow{2}{*}{$1$} & \multirow{2}{*}{$1$} & \multirow{2}{*}{$1$} & \multirow{2}{*}{$0$} & $\frac{d+1}{2a}$ & $a \geq d$ & \multirow{2}{*}{Cor. \ref{OneDefender}}   \\  \cline{5-6}
 &  &  &  & $\frac{2d-a+1}{2d}$ & $a \leq d$ & \\ \hline
\multirow{2}{*}{$2$} & \multirow{2}{*}{$1$} & \multirow{2}{*}{$1$} & \multirow{2}{*}{$1$} & $\frac{6a^2-2d^2-3d-1}{6a^2}$ & $a \geq d$ & \multirow{2}{*}{Prop. \ref{OneAttacker}}   \\  \cline{5-6}
 &  &  &  & $\frac{4a^2-3a-1}{6ad}$ & $a \leq d$ & \\ \hline
\multirow{2}{*}{$2$} & \multirow{2}{*}{$1$} & \multirow{2}{*}{$1$} & \multirow{2}{*}{$0$} & $\frac{2d^2+3d+1}{6a^2}$ & $a \geq d$ & \multirow{2}{*}{Cor. \ref{OneDefender}}   \\  \cline{5-6}
 &  &  &  & $\frac{6ad-4a^2+3a+1}{6ad}$ & $a \leq d$ & \\ \hline
\multirow{2}{*}{$3$} & \multirow{2}{*}{$1$} & \multirow{2}{*}{$1$} & \multirow{2}{*}{$1$} & $\frac{4a^3-d^3-2d^2-d}{4a^3}$ & $a \geq d$ & \multirow{2}{*}{Prop. \ref{OneAttacker}}   \\  \cline{5-6}
 &  &  &  & $\frac{3a^2-2a-1}{4ad}$ & $a \leq d$ & \\ \hline
\multirow{2}{*}{$3$} & \multirow{2}{*}{$1$} & \multirow{2}{*}{$1$} & \multirow{2}{*}{$0$} & $\frac{d^3+2d^2+d}{4a^3}$ & $a \geq d$ & \multirow{2}{*}{Cor. \ref{OneDefender}}   \\  \cline{5-6}
 &  &  &  & $\frac{4ad-3a^2+2a+1}{4ad}$ & $a \leq d$ & \\ \hline
\multirow{2}{*}{$1$} & \multirow{2}{*}{$2$} & \multirow{2}{*}{$1$} & \multirow{2}{*}{$1$} & $\frac{6ad-4d^2-3d+1}{6ad}$ & $a \geq d$ & \multirow{2}{*}{Prop. \ref{OneAttacker}}   \\  \cline{5-6}
 &  &  &  & $\frac{2a^2-3a+1}{6d^2}$ & $a \leq d$ & \\ \hline
\multirow{2}{*}{$1$} & \multirow{2}{*}{$2$} & \multirow{2}{*}{$1$} & \multirow{2}{*}{$0$} & $\frac{4d^2+3d-1}{6ad}$ & $a \geq d$ & \multirow{2}{*}{Cor. \ref{OneDefender}}   \\  \cline{5-6}
 &  &  &  & $\frac{6d^2-2a^2+3a-1}{6d^2}$ & $a \leq d$ & \\ \hline
\multirow{2}{*}{$2$} & \multirow{2}{*}{$2$} & \multirow{2}{*}{$2$} & \multirow{2}{*}{$2$} & $\frac{6a^2d-4ad^2-6ad-2a+2d^2+3d+1}{6a^2 d}$ & $a \geq d$ & \multirow{2}{*}{Prop. \ref{TwoAttacker}}   \\  \cline{5-6}
 &  &  &  & $\frac{2a^3-4a^2+a+1}{6a d^2}$ & $a \leq d$ & \\ \hline
\multirow{2}{*}{$2$} & \multirow{2}{*}{$2$} & \multirow{2}{*}{$2$} & \multirow{2}{*}{$1$} & $\frac{4a d^2+6a d+2a-2d^3-6d^2-4d}{6a^2d}$ & $a \geq d$ & \multirow{2}{*}{Equ. \ref{oneandone}}   \\  \cline{5-6}
 &  &  &  & $\frac{-2a^3+4a^2d+6a^2-6ad-4a+2d}{6a d^2}$ & $a \leq d$ & \\ \hline
\multirow{2}{*}{$2$} & \multirow{2}{*}{$2$} & \multirow{2}{*}{$2$} & \multirow{2}{*}{$0$} & $\frac{2d^3+4d^2+d-1}{6a^2d}$ & $a \geq d$ & \multirow{2}{*}{Prop. \ref{TwoDefender}}   \\  \cline{5-6}
 &  &  &  & $\frac{-4a^2 d -2 a^2 +6 a d^2+6 a d+3a-2d-1}{6 a d^2}$ & $a \leq d$ & \\ \hline
\multirow{2}{*}{$3$} & \multirow{2}{*}{$2$} & \multirow{2}{*}{$2$} & \multirow{2}{*}{$2$} & $\frac{12a^3d-6ad^3-12ad^2-12ad-6a+3d^3+10d^2+9d+2}{12a^3d}$ & $a \geq d$ & \multirow{2}{*}{Prop \ref{TwoAttacker}}   \\  \cline{5-6}
 &  &  &  & $\frac{6a^4-9a^3-2a^2+3a+2}{12a^2d^2}$ & $a \leq d$ & \\ \hline
\multirow{2}{*}{$3$} & \multirow{2}{*}{$2$} & \multirow{2}{*}{$2$} & \multirow{2}{*}{$1$} & $\frac{30ad^3+60ad^2+60ad+30a-12d^4-45d^3-70d^2-45d-8}{60a^3d}$ & $a \geq d$ & \multirow{2}{*}{Equ. \ref{oneandone}}   \\  \cline{5-6}
 &  &  &  & $\frac{-42a^4+60a^3d-60a^2d+75a^3-10a^2-15a-8}{60a^2d^2}$ & $a \leq d$ & \\ \hline
\multirow{2}{*}{$3$} & \multirow{2}{*}{$2$} & \multirow{2}{*}{$2$} & \multirow{2}{*}{$0$} & $\frac{6d^4+15d^3+10d^2-1}{30a^3 d}$ & $a \geq d$ & \multirow{2}{*}{Prop \ref{TwoDefender}}   \\  \cline{5-6}
 &  &  &  & $\frac{6a^4-30a^3d-15a^3+30a^2d^2+30a^2d+10a^2-1}{30a^2 d^2}$ & $a \leq d$ & \\ \hline
\end{tabular}
\caption{Generalized Engagement (single-attack) odds in \emph{RISK}}
\end{table}
\section{Acknowledgements} We would like to acknowledge an unpublished manuscript of Steven Miller \\
\url{http://web.williams.edu/go/math/sjmiller/public_html/projects/DieBattle.pdf} \\
that arrives at the same engagement odds as the $k=1$ case in this paper. The first and third authors implemented code to calculate conquer odds as part of a summer research project at Amherst College. Lastly the second and third authors analyzed conquer odds as the number of sides tend to infinity, arriving at a continuous version of conquer odds \cite{HM15}.  

\end{document}